\documentclass[12pt]{article}
\usepackage{amssymb}
\usepackage{latexsym,bm}
\usepackage{graphicx}
\usepackage{amsmath}
\usepackage{amsthm}
\usepackage{mathrsfs}

\newtheorem{thm}{Theorem}
\newtheorem{proof of theorem }{Proof of Theorem }

\newtheorem{nota}{Notation}
\newtheorem{defi}{Definition}

\newtheorem{lem}[thm]{Lemma}

\newcommand{\bea}{\begin{eqnarray*}}
\newcommand{\eea}{\end{eqnarray*}}
\newcommand{\be}{\begin{equation}}
\newcommand{\ee}{\end{equation}}
\newcommand{\ben}{\begin{eqnarray*}}
\newcommand{\een}{\end{eqnarray*}}

\date{}
\begin{document}
\title{Further results on the number of cliques in graphs covered by long cycles
\footnote{E-mail addresses:
{\tt mathdzhang@163.com}.}}
\author{\hskip -10mm Leilei Zhang\\
{\hskip -10mm \small Department of Mathematics, East China Normal University, Shanghai 200241, China}}\maketitle
\maketitle
\begin{abstract}
 Let $\Gamma(n,k)$ be the set of $2$-connected $n$-vertex graphs containing an edge that is not on any cycle of length at least $k+1.$ Let $g_s(n,k)$ denote the maximum number of $s$-cliques in a graph in
 $\Gamma(n,k).$ Recently, Ji and Ye [SIAM J. Discrete Math., 37 (2023) 917-924] determined $g_s(n,k).$ They remark that it is interesting to characterize the extremal graphs. In this paper,
 we give such a characterization.
\end{abstract}

{\bf Key words.} Extreme graph; Long Cycle; Clique

{\bf Mathematics Subject Classification.} 05C35, 05C38

\section{ Introduction}
We consider finite simple graphs, and use standard terminology and notation. The {\it order} of a graph is its number of vertices, and the {\it size} its number of edges. Let $e(G)$ denote the size of $G.$ Denote by $V(G)$ and $E(G)$ the vertex set and edge set of a graph $G.$ For graphs we will use equality up to isomorphism, so $G_1=G_2$ means that $G_1$ and $G_2$ are isomorphic. Let $\overline{G}$ denote the complement of $G.$ For two graphs $G$ and $H,$ $G\vee H$ denotes the {\it join} of $G$ and $H,$ which is obtained from the {\it disjoint union} $G+H$ by adding edges joining every vertex of $G$ to every vertex of $H.$  We denote by $kG$ the union of $k$ pairwise vertex disjoint copies of a graph $G.$ We always use $K_n$ to denote the complete graph of order $n.$

For a vertex $v$ in a graph, we denote by $d(v)$ and $N(v)$ the degree of $v$ and the neighborhood of $v$ in $G,$ respectively. For $H\subseteq V(G),$ we denote by $N_H(v)$ the set $H\cap N(v)$ and put $d_H(v)=|N_H(v)|.$ We denote by $\delta(G)$ the {\it minimum degree} of a graph $G.$  For two vertices $u$ and $v$, we use the symbol $u\leftrightarrow v$ to mean that $u$ and $v$ are adjacent and use $u\nleftrightarrow v$ to mean that $u$ and $v$ are non-adjacent. The {\it circumference} $c(G)$ of a graph $G$ is the length of a longest cycle in $G.$ A graph $G$ with order $n$ is {\it hamiltonian} if $c(G)=n,$ otherwise it is {\it nonhamiltonian}. An $s$-clique is a clique of cardinality $s.$ Let $c_e(G)$ be the maximum length of cycles in $G$ containing $e$.

\begin{nota}
For integers $n\ge 3$ and $c\ge 3,$ let $\alpha=\lfloor(n-1)/(c-1)\rfloor$ and $p=n-1-\alpha(c-1).$ Then we have $n-1=\alpha(c-1)+p.$  Let $H_{n,c}$ be a connected graph consisting of $\alpha$ blocks each of which is $K_c$ and a block which is $K_{p+1}$(note that each block of $G$ is $K_c$ when $p=0$). Denote by $h_s(n,c)$ the number of $s$-cliques in $H_{n,c};$ more precisely,
$$
  h_s(n,c)=\alpha\binom{c}{s}+\binom{p+1}{s}.
$$
\end{nota}
Let $G$ be a graph with $n$ vertices whose circumference $c(G)$ is no more
than $c$. In 1959, Erd\H{o}s and Gallai \cite{8} prove that the size of $G$ is no more than $\c(n-1)/2.$ The result is best possible when $n-1$ is divisible by $c-1$, in view of the graph consisting of $(n-1)/(c-1)$ copies of $K_c$ all having exactly one vertex in common. However, Woodall \cite{2} proved that, when $n-1$ is not divisible by $c-1$, the bound $c(n-1)/2$ can be decreased. Caccetta and Vijayan \cite{1} gave an alternative proof of the same result, and in addition, characterize the structure of the extremal graphs.
\begin{thm} (Cacetta and Vijayan \cite{1}, Woodall \cite{2})
Let $G$ be a graph on $n$ vertices and $n-1=\alpha(c-1)+p$, where $c\ge3,$ $\alpha\ge0$ and $0\le p<c-1$. If the circumference of $G$ is no more than $c,$ then
$$
e(G)\le h_2(n,c)
$$
with equality only if (i) $G=H_{n,c}$; (ii) $G$ is a connected graph consisting of $\alpha' (\alpha' <\alpha)$ blocks each of which is $K_c$ and a block which is $K_t\vee \overline{K_{n'-t}}$, with $c=2t$ and $p=c/2$ or $c/2-1$, where $t\ge2$ and $n'=n-\alpha'(c-1)$.
\end{thm}

\begin{nota} For integers $n\ge 3$ and $c\ge 3,$ let $\beta=\lfloor(n-2)/(k-2)\rfloor$ and let $q=n-2-\beta(k-2).$ Denote by $X_{n,c}=K_2\vee (\beta K_{k-2}+K_q)$ and let $g_s(n,k)$ be the number of $s$-cliques in $X_{n,k},$ more precisely,
$$
  g_s(n,c)=\begin{cases}
  \beta\binom{k}{2}+\binom{q+2}{2}-\beta, \,\,\,\, {\rm if}\,\,\,s=2;\\
  \beta\binom{k}{s}+\binom{q+2}{s},   \qquad\,\,\, {\rm if}\,\,\,s\ge3.
  \end{cases}
$$
\end{nota}

Let $G$ be a $2$-connected graph on $n$ vertices, and suppose there exists a pair of vertices in $G$ such that the length of the longest path between them is no more than $k$. Fan \cite{3} proved that the size of $G$ is no more than $(k+2)(n-2)/2.$ The result is best possible when $n-2$ is divisible by $k-2.$ Wang and Lv \cite{4} gave a sharpness results of Fan for all possible values $n$. The result of Fan \cite{3} and the result of Wang and Lv \cite{4} can be restated as follows.

\begin{thm} (Fan \cite{3}, Wang and Lv \cite{4})
For integers $n\ge 3$ and $k\ge 3$, let $G$ be a $2$-connected graph on $n$ vertices and suppose that $n-2=\beta(k-2)+q,$ where $\beta\ge 1$ and $0\le q\le k-3.$  If there exists an edge $e$ in $G$ such that $c_e(G)\le k$, then
$$
e(G)\le g_2(n,k)
$$
with equality if and only if (i) $G=X_{n,k}$; or
(ii) $G=K_2\vee(\beta'K_{k-2}+K_{t-1}\vee\overline{K_{n'-t}}),$ with $k=2t-2$ and $q=\frac{k-1}{2}$ or $q=\frac{k-3}{2},$ where $t\ge 3,$ $\beta'< \beta$ and $n'=n-\beta'(k-2).$
\end{thm}
Generalizing Theorem 1, Ji and Ye \cite{5} proved the following result.

\begin{thm}(Ji and Ye \cite{5})
For integers $n\ge 3$ and $c\ge s\ge 3$, let $G$ be a connected graph on $n$ vertices with the circumference of $G$ is no more than $c,$ then $N(K_s,G)\le h_s(n,c).$
\end{thm}

Note that if a $2$-connected graph $G$ contains an $s$-clique, then for every edge $e\in E(G),$ we have $c_e(G)\ge s.$ Generalizing Theorem 2, Ji and Ye \cite{5} proved the following theorem, which completely resolves a conjecture of Ma and Yuan \cite{9}.

\begin{thm}(Ji and Ye \cite{5})
For integers $n\ge 3$, $k\ge s\ge 3$, let $G$ be a $2$-connected graph on $n$ vertices. If there exists an edge $e$ in $G$ such that $c_e(G)\le k$, then $N(K_s,G)\le g_s(n,k).$
\end{thm}

Let $G$  be defined as Theorem 4. Ji and Ye \cite{5} remark that it is interesting to seek a characterization of the extremal graphs in Theorem 4. Inspired by these, we characterize the graphs $G$ that have the maximum number of $s$-cliques. The following two definitions will be used to describe our main results.

\begin{defi}
Let $G$ be a connected graph with $n$ vertices and the circumference $c(G)$ is no more than $c.$ The graph $L(G,c)$ is called an edge maximal graph with respect to $G$ and $c$ if for any $e\in E(\overline{L(G,c)}),$ $c(L(G,c)+e)>c.$
\end{defi}

\begin{defi}
Let $G$ be a $2$-connected graph on $n$ vertices and suppose $G$ contains an edge $uv$ with $c_{uv}(G)\le k.$ The graph $M(G,uv)$ is called an edge maximal graph with respect to $G$ and $uv$ if for any $e\in E(\overline{M(G,uv)}),$ $c_{{uv}}(M(G,uv)+e)> k.$
\end{defi}

The following are the main results of these paper.

\begin{thm}
For integers $n\ge 3$ and $c\ge s\ge 3,$ let $G$ be a connected $n$-vertex graph  with $c(G)\le c$ and let $\alpha=\lfloor(n-1)/(c-1)\rfloor.$ If $N_s(G)=h_s(n,c)$, then for $s \leq p+1,$ $G$ is isomorphic to the graph $H_{n,c}$. Otherwise, $L(G,c)$ is a connected graph composed of $\alpha$ blocks, each of which is $K_c$, while the remaining blocks are complete graphs.
\end{thm}
\begin{thm}
For integers $n\ge 3$ and $k\ge s\ge 3$, let $G$ be a $2$-connected graph on $n$ vertices and suppose $G$ has an edge $e$ with $c_e(G)\le k.$ If $N_s(G)=g_s(n,k),$ then $M(G,e)=X_{n,k}.$
\end{thm}

\section{Proof of the main results}
 We will need the following notations and lemmas.
\begin{nota}
Fix $n-1 \geq c\geq 2k\ge 4.$ Let $F(n,c,k)=K_k\vee(K_{c+1-2k}+\overline{K_{n-c-1+k}})$.
Denote by $f_s(n,c,k)$ the number of $s$-cliques in $F(n,c,k)$; more precisely,
$$
  f_s(n,c,k)=\binom{c+1-k}{s}+(n-c-1+k)\binom{k}{s-1}.
$$
\end{nota}
In \cite{7}, Luo determined the maximum number of $s$-cliques in a $2$-connected nonhamiltonian graph of order $n.$

\begin{thm}\label{luo} (Luo \cite{7})
Let $n>c\ge 4,$ if $G$ be a $2$-connected  graph of order $n$ with circumference $c,$ then $N_s(G)\le {\rm max}\{f_s(n,c,2),f_s(n,c,\lfloor c/2 \rfloor)\}.$
\end{thm}

\begin{lem}
Let $n>c\ge 4$ and $c\ge s\ge 3,$ then ${\rm max}\{f_s(n,c,2),f_s(n,c,\lfloor c/2 \rfloor)\}<h_s(n,c).$
\end{lem}
\begin{proof}
It is enough to prove that $f_s(n,c,2)<h_s(n,c)$ and $f_s(n,c,\lfloor c/2 \rfloor)<h_s(n,c).$
Let $n-1=\alpha(c-1)+p,$ where $\alpha\ge1$ and $0\le p\le c-2.$ Recall that $h_s(n,c)=\alpha\binom{c}{s}+\binom{p+1}{s}.$

We first prove $f_s(n,c,2)<h_s(n,c).$ Note that if $s\ge 4,$ the result hold. If $s=3,$ we have
\begin{align*}
h_3(n,c)-f_3(n,c,2)=&\alpha\binom{c}{3}+\binom{p+1}{3}-\binom{c-1}{3}-(n-c+1)\\
=&(\alpha-1)\left[\binom{c}{3}-(c-1)\right]+\binom{c-1}{2}+\binom{p+1}{3}-(p+1)>0.
\end{align*}

Now we will prove $f_s(n,c,\lfloor c/2 \rfloor)<h_s(n,c).$ Assume that $c$ is even and $c=2t$; then $F(n,c,t)=K_t\vee \overline{K_{n-t}}.$ We begin by prove the following inequality
\begin{equation}\label{1}
  \binom{t}{s}+(2t-1)\binom{t}{s-1}< \binom{2t}{s}.
\end{equation}

The two cases $s=3$ and $s\ge t+1$ can be easily check. So we assume $s\ge4$ and $t\ge s.$  Inequality (\ref{1}) is equivalent to
$$
\frac{t(t-1)...(t-s+1)}{s!}+\frac{s(2t-1)\cdot t(t-1)...(t-s+2)}{s!}< \frac{2t(2t-1)...(2t-s+1)}{s!}.
$$
Simplifying the above inequality, we obtain
\begin{equation}\label{2}
(2s+1)t-2s+1<\frac{2t}{t}\cdot \frac{2t-1}{t-1}...\frac{2t-s+2}{t-s+2}\cdot(2t-s+1).
\end{equation}

Since $\frac{2t-i}{t-i}\geq2$ for any $0\leq i\leq t-1$, the right hand side of (\ref{2}) is no less than $2^{s-1}(2t-s+1)$. Note also that $t\geq s$; thus it suffices to show
$$
(2^s-2s-1)s> s2^{s-1}-2^{s-1}-2s+1,
$$
which is equivalent to $(s+1)2^{s-1}> 2s^2-s+1$. Recall that $s\geq4$; thus it is sufficient to prove $5\cdot2^{s-1}> 2s^2-3$, which can be easily checked.

Let $n-t=\alpha_1(c-1)+p_1$ with $0\le p_1\le c-2,$  and we have $n-1=\alpha_1(c-1)+p_1+t-1.$
If $t+p_1-1\le c-2,$ we have $\alpha_1\ge1.$ Then
\begin{align*}
  f_s(n,c,t)&=\left[\alpha_1(c-1)+p_1\right]\binom{t}{s-1}+\binom{t}{s}\\
  &=\alpha_1(2t-1)\binom{t}{s-1}+\binom{t}{s}+p_1\binom{t}{s-1}\\
  &<\alpha_1\binom{2t}{s}+\binom{t+p_1}{s}=h_s(n,c),
\end{align*}
where the inequality follows from the inequality (\ref{1}) and the fact $\alpha_1\ge1.$ If $c-1\le t+p_1-1\le 3t-2,$ then
\begin{align*}
  f_s(n,c,t)&=\left[\alpha_1(c-1)+p_1\right]\binom{t}{s-1}+\binom{t}{s}\\
  &=\alpha_1(2t-1)\binom{t}{s-1}+p_1\binom{t}{s-1}+\binom{t}{s}\\
  &\le\alpha_1\binom{2t}{s}+(2t-2)\binom{t}{s-1}+\binom{t}{s}\\
  &<\alpha_1\binom{2t}{s}+\binom{2t}{s}+\binom{t+p_1-c+1}{s}=h_s(n,c),
\end{align*}
where the second inequality follows from the inequality (\ref{1}).

Assume that $c=2t+1$ is odd (the proof is similar to the case when $c$ is even); then $F(n,c,t)=K_t\vee(K_2+K_{n-t-2})$. We first prove the following inequality
\begin{equation}\label{3}
  \binom{t+2}{s}+2t\binom{t}{s-1}< \binom{2t+1}{s}.
\end{equation}
The case $s=3$ can be easily checked. So we assume $s\ge4.$  Note that
$$
\binom{t+2}{s}+2t\binom{t}{s-1}\le \binom{t+2}{s}+2t\binom{t+1}{s-1}.
$$
So, it suffices to prove
\begin{equation*}
 \binom{t+2}{s}+2t\binom{t+1}{s-1}<  \binom{2t+1}{s},
\end{equation*}
which is equivalent to
\begin{equation*}
\frac{(t+2)(t+1)\ldots(t-s+3)}{s!}+\frac{s(2t)(t+1)\ldots(t-s+3)}{s!}<
\frac{(2t+1)(2t)\ldots(2t-s+2)}{s!}.
\end{equation*}
Simplifying the above inequality, we obtain
\begin{equation}\label{4}
t+2+2ts<\frac{2t+1}{t+1}\cdot\frac{2t}{t}\cdot\frac{2t-1}{t-1}\cdots\frac{2t-s+3}{t-s+3}\cdot(2t-s+2).
\end{equation}
Since $\frac{2t-i}{t-i}>2$ for any $0\leq i\leq t-1$ and
$\frac{2t+1}{t+1}\cdot\frac{2t}{t}\cdot\frac{2t-1}{t-1}>8$, the right hand side of (\ref{4}) is no less than $2^{s-1}(2t-s+2)$. Note also that $t\geq s$; thus it suffices to show
$$
(2^s-2s-1)s>(s-2)\cdot2^{s-1}+2,
$$
which is equivalent to $(s+2)2^{s-1}> 2s^2+s+2$. Recall that $s\geq4$; thus it is sufficient to prove $6\cdot2^{s-1}> 2s^2+s+2$, which can be easily checked.

Let $n-t-2=\alpha_2(c-1)+p_2$ with $0\le p_2\le c-2,$  and we have $n-1=\alpha_2(c-1)+p_2+t+1.$
If $t+p_2+1\le c-2,$ we have $\alpha_2\ge1.$ Then
\begin{align*}
  f_s(n,c,t)&=\left[\alpha_2(c-1)+p_2\right]\binom{t}{s-1}+\binom{t+2}{s}\\
  &=\alpha_2\cdot2t\binom{t}{s-1}+\binom{t+2}{s}+p_2\binom{t}{s-1}\\
  &<\alpha_2\binom{2t+1}{s}+\binom{t+p_1+2}{s}=h_s(n,c),
\end{align*}
where the inequality follows from the inequality (\ref{3}) and the fact $\alpha_2\ge1.$ If $c-1\le t+p_2+1\le 3t-1,$ then
\begin{align*}
  f_s(n,c,t)&=\left[\alpha_2(c-1)+p_2\right]\binom{t}{s-1}+\binom{t+2}{s}\\
  &=\alpha_2\cdot2t\binom{t}{s-1}+p_2\binom{t}{s-1}+\binom{t+2}{s}\\
  &\le\alpha_2\binom{2t+1}{s}+(2t-1)\binom{t}{s-1}+\binom{t+2}{s}\\
  &<\alpha_2\binom{2t+1}{s}+\binom{2t+1}{s}+\binom{t+p_2-c+2}{s}=h_s(n,c),
\end{align*}
where the second inequality follows from the inequality (\ref{3}). This completes the proof of Lemma 8.
\end{proof}

Now we are ready to prove Theorem 5.

\noindent{\bf Proof of Theorem 5.}
Let $G$ be a connected $n$-vertex graph with $n\geq3$ and $c(G)\le c$, Suppose that $N_s(G)=h_s(n,c).$ Let $L(G,c)$ be the edge maximal graph with respect to $G$ and $c$. We will use induction on $n$. The result holds trivially for $n=3$ or $c=3$. So next we assume that $n\geq 4$ and $c\ge 4.$

First, we consider the case that $n\le c.$ Note that $N_s(G)=h_s(n,c)=\binom{n}{s}.$ If $s\le n,$ then $G$ is a complete graph $K_n.$ However, if $s\ge n+1,$ then $N_s(G)=\binom{n}{s}=0.$ Since $L(G,c)$ is edge-maximal with respect to $G$ and $c$, it follows that $L(G,c)$ is a complete graph. We then consider the case that $n\ge c+1.$ We divide the proof in the following two cases.

{\bf Case 1.} $s\le p+1.$ If $G$ is $2$-connected, by Theorem 7 and Lemma 8 we have
$$
N_s(G)\le {\rm max}\{f_s(n,c,2),f_s(n,c,\lfloor c/2 \rfloor)\}<h_s(n,c),
$$
which is a contradiction. Then $G$ has a cut-vertex $v.$ Let $H$ be a connected component of $G-v$, and let $G_1=G[H \cup  \{ v\} ]$ and $G_2=G-V(H)$. Then both $G_1$ and $G_2$ are connected and $G_1\cap G_2=\{v\}.$ For each $i\in [2]$, let $n_i=|V(G_i)|$, and assume
$n_i-1= \alpha_i(c-1)+p_i$ with $0\leq p_i\leq c-2$. Then $n=n_1+n_2-1$ and $N_s(G)=N_s(G_1)+N_s(G_2)$. For $s\ge 3,$ the following inequality holds:
$$
 \binom{p_1+1}{s}+\binom{p_2+1}{s}\le \begin{cases}
   \binom{p_1+p_2+1}{s},\qquad\qquad\quad\,\,{\rm if}\,\,\, p_1+p_2\le c-2;\\
   \binom{c}{s}+\binom{p_1+p_2-c+2}{s},\qquad {\rm if}\,\, c-1\le p_1+p_2\le 2(c-2).
 \end{cases}
$$
Equality holds if and only if $p_1+p_2\le c-2$ and $p_1=0$ or $p_2=0.$ Note that
\begin{align*}
N_s(G)&=N_s(G_1)+N_s(G_2)\le h_s(n_1,c)+h_s(n_2,c)\\
&=\alpha_1\binom{c}{s}+\binom{p_1+1}{s}+\alpha_2\binom{c}{s}+\binom{p_2+1}{s}\\
&\le
\begin{cases}
   (\alpha_1+\alpha_2)\binom{c}{s}+\binom{p_1+p_2+1}{s}\qquad\qquad\,\,\,\,{\rm if}\,\,\, p_1+p_2\le c-2;\\
   (\alpha_1+\alpha_2+1)\binom{c}{s}+\binom{p_1+p_2-c+2}{s}\qquad {\rm if}\,\, c-1\le p_1+p_2\le 2(c-2).
 \end{cases}\\
&=h_s(n,c).
\end{align*}
Since $N_s(G)=h_s(n,c),$ we deduce that $p_1+p_2\le c-2$ and $p_1=0$ or $p_2=0.$ Without loss of generality, we assume that $p_1=0.$ Applying the inductive hypothesis to $G_1$ and $G_2,$ we can conclude that $G_1$ consists of $\alpha_1$ copies of $K_c$ and $G_2$ consists of $\alpha_2$ copies of $K_c$ and one $K_{p_2+1}$. Since $n-1=(\alpha_1+\alpha_2)(c-1)+p_2,$ we have $G=H_{n,c}.$

{\bf Case 2.} $s\ge p+2.$ We consider the graph $L(G,c).$ Note that adding edges to $G$ does not reduce the number of $s$-cliques. Then $N_s(G)=N_s(L(G,c))=h_s(n,c).$ If $L(G,c)$ is $2$-connected, we will get a contradiction by Theorem 7 and Lemma 8. Then $L(G,c)$ has a cut-vertex $v.$ Similar to Case 1, we define $G_1$ and $G_2$ by considering the graph $L(G,c).$ Since $L(G,c)$ is edge maximal graph with respect to $G$ and $c.$ We deduce that $G_1$ and $G_2$ are also edge maximal. For $s\ge 3,$ the following inequality holds:
$$
 \binom{p_1+1}{s}+\binom{p_2+1}{s}\le \begin{cases}
   \binom{p_1+p_2+1}{s}=0\qquad\,\,\,{\rm if}\,\,\, p_1+p_2\le c-2;\\
   \binom{c}{s}+\binom{p_1+p_2-c+2}{s}\quad {\rm if}\,\, c-1\le p_1+p_2\le 2(c-2).
 \end{cases}
$$
Equality holds if and only if $p_1+p_2\le c-2.$  Applying the inductive hypothesis to each $G_i$, we have
\begin{align*}
N_s(L(G,c))&=N_s(G_1)+N_s(G_2)\le h_s(n_1,c)+h_s(n_2,c)\\
&=\alpha_1\binom{c}{s}+\binom{p_1+1}{s}+\alpha_2\binom{c}{s}+\binom{p_2+1}{s}\\
&\le
\begin{cases}
   (\alpha_1+\alpha_2)\binom{c}{s}+\binom{p_1+p_2+1}{s}\qquad\qquad\,\,{\rm if}\,\,\, p_1+p_2\le c-1\\
   (\alpha_1+\alpha_2+1)\binom{c}{s}+\binom{p_1+p_2-c+2}{s}\qquad {\rm if}\,\, c\le p_1+p_2\le 2(c-2)
 \end{cases}\\
&=h_s(n,c)
\end{align*}
Since $N_s(L(G,c))=h_s(n,c),$ we deduce that $p_1+p_2\le c-2.$ Note that $s\ge p_1+p_2+2.$ we have $s\ge p_1+2$ and $s\ge p_2+2$. Then $G_i,$ $i=1,2$ is a connected graph consisting of $\alpha_i$ blocks, each of which is $K_c$, while the remaining blocks are complete graphs. Then $L(G,c)$ is a connected graph consisting of $\alpha$ blocks, each of which is $K_c$, while the remaining blocks are complete graphs. \hfill $\Box$

\begin{defi}
For a given edge $uv$ in $G$, an {\it edge-switching} from $v$ to $u$ is to replace each edge $vw$ by a new edge $uw$ for every $w\in N(v)\setminus N[u]$. The resulting graph is called the {\it edge-switching graph} of $G$ from $v$ to $u$, denoted by $G[v\rightarrow u]$.
\end{defi}

Let $ux$ be an edge of $G$ and suppose that $\{u,x\}$ is not a vertex cut of $G$.  By Lemma 2.3 and Lemma 2.5 in \cite{5}, we have that if $N(u)\cap N(x)=\emptyset,$ then the graph $G/ux$ is $2$-connected and  $N_s(G/ux)\ge N_s(G)$ for $s\ge 3.$ If $N(u)\cap N(x)\neq\emptyset,$ then $G[x\rightarrow u]$ is $2$-connected and $N_s(G[x\rightarrow u])\ge N_s(G).$ Let $v$ be a neighbor of $u$ in $G$ and $v\neq x$. If $c_{uv}(G)\le k,$ by Lemma 2.4 of \cite{5}, we have $c_{uv}(G[x\rightarrow u])\le k.$

\begin{lem}
For $n\ge k+1,$ let $G$ be a $2$-connected $n$-vertex graph with $N_s(G)=g_s(n,k),$ and let $\{v,x\}$ be two neighbors of $u$ in $G$ with $c_{uv}(G)\leq k$ and $N_G(u)\cap N_G(x)=\emptyset.$ If $N_s(G/ux)=N_s(G)$ and $M(G/ux,uv)=X_{n-1,k},$ then $\{u,v\}$ is a $2$-vertex cut of $G.$
\end{lem}

\noindent{\bf Proof.} Let $G'=M(G/ux).$ Since $c_{uv}(G)\le k,$ we can conclude that the length of a longest cycle containing $uv$ in $G'$ is at most $k,$ i.e.,  $c_{uv}(G')\le k.$ It is worth noting that $N_s(G')\le g_s(n-1,k)=N_s(X_{n-1,k})$ and $N_s(G')=N_s(G)=g_s(n,k).$  Therefore, we have $N_s(G')=g_s(n-1,k)=N_s(X_{n-1,k}).$ Suppose $(n-1)-2=\beta(k-2)+q, $ where $\beta\ge 1$ and $0\le q\le k-3.$ By the definition of $M(G',uv),$ the graph $G'$ is a spanning subgraph of $X_{n-1,k}.$ Let $B'_i,$ $1\le i\le\beta,$ be the connected component of $X_{n-1,k}-\{u,v\}$ with $k-2$ vertices. Furthermore, define $B_i=X_{n-1,k}[V(B'_i)\cup\{u,v\}]$, $1\le i\le \beta.$ Then $V(B_i)$ forms a clique with $k$ vertices. Since the edges in $B_i,$ $1\le i\le\beta$ are contained in some $K_s$ in $X_{n-1,k}$ and $N_s(G')=N_s(X_{n-1,k}),$ we conclude that $G'$ contains $B_i,$ $1\le i\le\beta$. Let us denote by $B_0=G'-\sum_{i=1}^{\beta}B'_i$ and $B'_0=B_0-\{u,v\}.$

Suppose that $\{u,v\}$ is not a $2$-vertex cut in $G$. In this case, we have $N_G(x)\cap V(B'_i)\neq\emptyset, 0\le i\le \beta.$  We then consider the graph $T_i=G[V(B_i)\cup x],$ $0\le i\le \beta.$ Obviously, the graph $T_i, 1\le i\le \beta$ contains a $(k-1)$-clique $V(B_i')\cup \{v\}.$ Recall that $N_G(u)\cap N_G(x)=\emptyset.$ Thus, we have $N_{T_i}(u)\cup N_{T_i}(x)=V(B_i')\cup \{v\}$ and $N_{T_i}(u)\cap N_{T_i}(x)=\emptyset,$ $1\le i\le \beta.$ It is easy to verify that there exists a cycle of length $k+1$ that contains $uv$ in $T_i,$ where $1\le i\le \beta,$  a contradiction to $c_{uv}(G)\le k.$ This completes the proof of Lemma 9.
\hfill $\Box$

\begin{lem}
For $n\ge k+1,$ let $G$ be a $2$-connected $n$-vertex graph with $N_s(G)=g_s(n,k),$ and let $\{v,x\}$ be two neighbors of $u$ in $G$ with $c_{uv}(G)\leq k$ and $N_G(u)\cap N_G(x)\neq\emptyset.$ If $N_s(G[x\rightarrow u])=N_s(G)$ and $M(G[x\rightarrow u],uv)=X_{n,c},$  then $\{u,v\}$ is a $2$-vertex cut of $G.$
\end{lem}

\noindent{\bf Proof.} Let $Q=G[x\rightarrow u].$ Since $M(Q,uv)=X_{n,k}$ and $N_s(Q)=g_s(n,k).$ The vertex set $\{u,v\}$ is a $2$-vertex cut of $Q.$ Suppose $(n-2)=\alpha(k-2)+p,$ where $\alpha\ge 1$ and $0\le p \le k-3.$ Let $Q'_i,$ $1\le i\le\beta,$ be the connected component of $X_{n-1,k}-\{u,v\}$ with $k-2$ vertices. Furthermore, define $Q_i=X_{n-1,k}[V(Q'_i)\cup\{u,v\}]$, $1\le i\le \beta.$ Then $V(Q_i)$ forms a clique with $k$ vertices. Since the edges in $Q_i,$ $1\le i\le\beta$ are contained in some $K_s$ in $X_{n-1,k}$ and $N_s(Q)=N_s(X_{n,k}),$ we conclude that $Q$ contains $Q_i,$ $1\le i\le\beta$. Let us denote by $Q_0=Q-\sum_{i=1}^{\beta}Q'_i$ and $Q'_0=Q_0-\{u,v\}.$ Suppose that $\{u,v\}$ is not a $2$-vertex cut in $G$. We divide the proof in the following two cases.

{\bf Case 1.} If $x\in V(Q_i'),$ $1\le i\le \alpha,$ without loss of generality, we assume that $x\in V(Q_1').$ Then there exist some vertices in $Q'_i,$ $i=0,2,\ldots,\alpha$ adjacent to $x$ in $G.$ Since $N_{Q_1}(x)=N_{Q_1}(u)$ and by the definition of edge-switching, we conclude that the graph $G$ contains $Q_1$ as a subgraph. We then consider the graphs $H_i=G[V(Q_i)\cup \{x\}],$ $i=0,2,\ldots,\alpha$ in $G.$  Note that $Q_i$ is $2$-connected in $Q,$ and by de definition of edge-switching, we have that graphs $H_i,$ $i=0,2,\ldots,\alpha$ are also $2$-connected.

Since $n\ge k+1,$ there exists an $i\in\{0,2,\ldots,\alpha\}$ such that $Q'_i$ is not an empty graph. Without loss of generality, we assume that $i=0.$ Note that $\{u,v\}$ is not a $2$-vertex cut in $G$. The vertex $x$ has neighbors in $Q'_0$. Denote by $x'$ one of the neighbors of $x$ in $Q_0.$ For the edges $uv$ and $xx',$ it follows from the $2$-connectivity of $H_0$ that $H_0$ has a cycle $C_0$ containing both $uv$ and $xx'.$ Then there exists a path $P_1$ in $C_0$ connecting $x'$ and $v$ such that $x$ and $u$ are not in $V(P_1),$ or there exists a path $P_2$ in $C_0$ connecting $x'$ and $u$ such that $x$ and $v$ are not in $V(P_2).$ If $P_1$ exists, let $C_1'$ be a hamiltonian cycle of $H_1-v$ such that $xu\in E(C_1').$ Then
$$
c_{uv}(G)\ge ||E(C_1')-xu+uv+E(P_1)+xx'||\ge k-1-1+1+1+1=k+1.
$$
If $P_2$ exists, let $C_1''$ be a hamiltonian cycle of $H_1-u$ which contains $xv.$ Then
$$
c_{uv}(G)\ge ||E(C_1'')-xv+xx'+E(P_2)+uv||\ge k-1-1+1+1+1=k+1.
$$
Then there exists a cycle of length $k+1$ that contains $uv.$ a contradiction to $c_{uv}(G)\le k.$

{\bf Case 2.} If $x\in B_0',$ in this case, we define $H_i=G[V(Q_i)\cup\{x\}],$ $1\le i\le \alpha$ and $H_0=G[V(Q_i)].$ By the definition of edge-switching, we have $N_{H_1}(x)\cup N_{H_1}(u)=V(H_1),$ $N_{H_1}(x)\cap N_{H_1}(u)=\emptyset$ and $V(H_1)-\{x,u\}$ is a clique. It can be check that every edge of $H_1$ contained in a cycle in $H_1$ with length at least $k+1.$ Thus, we have that $c_{uv}(G)\ge k+1,$ a contradiction. This completes the proof of Lemma 10\hfill $\Box$

Now we are ready to prove Theorem 6.

\noindent{\bf Proof of Theorem 6.} Let $G$ be a $2$-connected graph on $n$ vertices and let $G$ contain an edge $e$ with $c_e(G)\le k.$ We will use induction on $n.$ The result holds trivially for $n=3$ or $k=3$. Thus we assume that $n\ge4$ and $k\ge4.$ Note that if $n\le k,$ the result holds trivially. therefore, we assume $n\ge k+1.$ We divide the proof in the following two cases.

{\bf Case 1.} The graph $G$ has a $2$-vertex cut $\{x,y\}$ such that $xy$ is an edge.

Let $W$ be a connected component of $G-\{x,y\}$. Furthermore, let $G_1=G[V(W)\cup
\{x, y\}]$ and $G_2 =G-V(W)$. Then both $G_1$ and $G_2$ are $2$-connected. For convenience,
let $|V(G_i)|=n_i\geq3$, $n_i-2=\beta_i(k-2)+q_i$ and $0\leq q_i \leq k-3$ for $i\in  [2]$.

If for any edge $e_1\in E(G_1),$ we have $c_{e_1}(G_1)\ge k+1.$  Let $e_2$ be an edge in $E(G_2)-\{xy\}.$ Since $G_2$ is $2$-connected, $G_2$ has a cycle $C_2$ containing $e_2$ and $xy.$ Note that $c_{xy}(G_1)\ge k+1.$ $G_1$ has a cycle $C_1$ with length at least $c+1$ containing $xy.$ Then the cycle $C_1+C_2-\{xy\}$ containing $e_2$ and has length at least $k+1.$ Thus, $c_e(G)\ge k+1$ for any edge in $G,$ a contradiction. Hence $G_1$ has an edge $e_1$ such that $c_{e_1}(G_1)\le k.$ By similar discussions, there exists an edge $e_2$ such that $c_{e_2}(G_2)\le k.$  For $s\ge3$, the following inequality holds:
\begin{equation*}\label{5}
 \binom{q_1+2}{s}+\binom{q_2+2}{s}
  \le
\begin{cases}
   \binom{q_1+q_2+2}{s},\qquad\qquad\,\,{\rm if}\,\,\, q_1+q_2\le k-3;\\
   \binom{k}{s}+\binom{q_1+q_2-k+4}{s},\,\,\,\,\,\, {\rm if}\,\, k-2\le q_1+q_2\le 2(k-3).
 \end{cases}\\
\end{equation*}
If $s\le q_1+q_2+2,$ equality holds if and only if $q_1+q_2\le k-3$ and $q_1=0$ or $q_2=0.$ If $s\ge q_1+q_2+3,$ equality holds if and only if $q_1+q_2\le k-3.$ Note that
\begin{align*}
  N_s(G)&=N_s(G_1)+N_s(G_2)\le g_s(n_1,k)+g_s(n_2,k)\\
  &=\beta_1\binom{k}{s}+\binom{q_1+2}{s}+\beta_2\binom{k}{s}+\binom{q_2+2}{s}\\
  &=(\beta_1+\beta_2)\binom{k}{s}+\binom{q_1+2}{s}+\binom{q_2+2}{s}\\
  &\le
\begin{cases}
   (\beta_1+\beta_2)\binom{k}{s}+\binom{q_1+q_2+2}{s}\qquad\qquad\,\,{\rm if}\,\,\, q_1+q_2\le k-3;\\
   (\beta_1+\beta_2+1)\binom{k}{s}+\binom{q_1+q_2-c+4}{s}\qquad {\rm if}\,\, k-2\le q_1+q_2\le 2(k-3).
 \end{cases}\\
 &=g_s(n,k).
\end{align*}
Since $N_s(G)=g_s(n,k),$ we have $N_s(G_i)=g_s(n_i,k),i=1,2$ and $q_1+q_2\le k-3.$ Note that $xy\in E(G_1)\cap E(G_2).$ We have $C_{xy}(G_1)\le k$ or $C_{xy}(G_2)\le k.$  Without loss of generality, we assume that $C_{xy}(G_1)\le k.$

Since $q_1+q_2\le k-3$ and $n\ge k+1,$ there exists an $i\in \{1,2\}$ such that $\beta_i\ge 1.$ By the inductive hypothesis, we have $M(G_i,e_i)=K_2\vee(\beta_iK_{k-2}+K_{q_i}),$ where $e_i$ is an edge of $G_i$ with $c_{e_i}(G_i)\le k,$ $i=1,2.$ Then $G_i$ is a subgraph of $K_2\vee(\beta_iK_{k-2}+K_{q_i}),$ $i=1,2.$ Since $N_s(G_i)=g_s(n_i,k),$ we have that $K_2\vee(\beta_iK_{k-2})$ is a subgraph of $G_i,$ $i=1,2.$ Let $ab$ denote $K_2$ of $K_2\vee(\beta_2K_{k-2})$ in $G_2.$ If $\beta_1=0.$ We can deduce that $xy\in E(G_2)-E(K_2\vee(\beta_2K_{k-2}))+\{ab\}.$  Otherwise, we will get a contradiction that every edge of $G$ is covered by a cycle of length at least $k+1.$ Then $M(G,e)=X_{n,k}.$ If $\beta_1\ge 1,$ then $c_{xy}(G_2)\le k.$ Otherwise, we can deduce that every edge of $G$ is covered by a cycle of length at least $k+1,$ a contradiction. Then we have $M(G,e)=X_{n,k}.$



{\bf Case 2.} The graph $G$ does not have a $2$-vertex cut $\{x,y\}$ such that $xy$ is an edge.

In this case, we assume that $G$ is a $2$-connected $n$-vertex graph with $N_s(G)=g_s(n,k),$ and $G$ does not have a $2$-vertex $\{x,y\}$ such that $xy$ is an edge. Denote
$$
\xi(G)= {\rm max}\,\{d_G(u)|\,u\,\,{\rm is}\,\,{\rm an}\,\,{\rm end}\,\,{\rm vertex}\,\,\,{\rm of}\,\,{\rm some}\,\,{\rm edge}\,\, e\,\,{\rm with}\,\, c_e(G)\le k\}.
$$
Among all the graphs $G$, choose $Q$ such that $\xi(G)$ is as large as possible.  Now, we will proceed to demonstrate that such a graph $Q$ cannot exist. Choose an edge $uv$ of $Q$ such that $c_{uv}(Q)\le k$ and $d_Q(v)\leq d_Q(u)=\xi (Q)$.

{\bf Subcase 2.1.} The vertex $u$ is a dominating vertex of $Q.$

By the Claim 3 of proof of Theorem 5 in \cite{5}, we have that the circumference of $Q-u$ is no more than $k-1$. Suppose $(n-1)-1=\alpha(k-2)+p$ and $0\le p\le k-3.$ Then
\begin{align*}
 N_s(Q)&=N_s(Q-u)+N_{s-1}(G-u)\\
 &\le \alpha\binom{k-1}{s}+\binom{p+1}{s}+ \alpha\binom{k-1}{s-1}+\binom{p+1}{s-1}\\
 &=\alpha\binom{k}{s}+\binom{p+2}{s}
 =g_s(n,k),
\end{align*}
equality holds if and only if $N_s(Q-u)=\alpha\binom{k-1}{s}+\binom{p+1}{s}$ and $N_{s-1}(G-u)=\alpha\binom{k-1}{s-1}+\binom{p+1}{s-1}.$ By Theorem 5, $Q-u$ has a cut vertex $x$. Then $ux$ is an edge and $\{u,x\}$ is a vertex cut of $Q,$  a contradiction.

{\bf Subcase 2.2.} The vertex $u$ is not a dominating vertex of $Q.$

Let $p=uu_1\ldots u_k$ be defined as in the Claim 2 of proof of Theorem 5 in \cite{5}. If $N(u)\cap N(u_1)=\emptyset,$ let $Q'=Q/uu_1.$ Since $\{u,u_1\}$ is not a $2$-vertex cut. We have that $Q'$ is $2$-connected, $c_{uv}(Q')\le k$ and $N_s(Q')=g_s(n,k).$ By the inductive hypothesis, we have $M(Q',uv)=X_{n-1,k}.$ By Lemma 9, we can deduce that $\{u,v\}$ is a 2-vertex cut of $G$ and $uv$ is an edge, a contradiction.

If $N(u)\cap N(u_1)\neq\emptyset,$ let $Q''=Q[u_1\rightarrow u].$ We have that $Q''$ is $2$-connected, $c_{uv}(Q'')\le k$ and $N_s(Q'')=g_s(n,k).$ If $Q''$ have a vertex cut $\{x,y\}$ such that $xy$ is an edge of $Q''.$ By Case 1, we have $M(G'',uv)=X_{n,k}.$ Then $Q$ has a $2$-vertex cut $\{u,v\}$ by Lemma 10, a contradiction. If $Q''$ does not have an edge $xy$ such that $\{x,y\}$ is a vertex cut. Then $\xi(Q'' )\geq d_{Q''}(u)>d_Q(u)=\xi (Q)$, which contradicts the maximality of $\xi (Q)$. \hfill $\Box$

\vskip 5mm
{\bf Acknowledgement.} The author is grateful to Professor Xingzhi Zhan for his constant support and guidance. This research was supported by the NSFC grant 12271170 and Science and Technology Commission of Shanghai Municipality (STCSM) grant 22DZ2229014.

\end{document}